\documentclass[10pt,a4paper]{article}
\usepackage[latin2]{inputenc}
\usepackage[T1]{fontenc}
\usepackage{amsmath,amssymb,multicol,amsthm}
\usepackage{textcomp,array,amsfonts}
\usepackage{graphicx,color,mathrsfs,txfonts}
\usepackage[all]{xy}
\theoremstyle{plain}
\newtheorem{thm}{\bfseries Theorem}[section]
\newtheorem{lem}[thm]{\bfseries Lemma}

\newtheorem{cor}[thm]{\bfseries Corollary}
\date{}
\title{Morphisms Between Classical and Infinitesimal Polylogarithmic and Grassmannian Complexes}
\author{Raziuddin Siddiqui \thanks{email: rdsiddiqui@fuuast.edu.pk}\\ \small \textit{Mathematical Sciences Research Centre}\\ 
\small \textit{Federal Urdu University, Karachi.}}

\begin{document}
\maketitle
\begin{abstract}
In this paper we want to introduce two commutative diagrams for weight $n$=2 and $n$=3 with six faces on each. These diagrams describe the relations between Grassmannian complex in geometric configurations, Bloch-Suslin's complex for weight $n$=2 and Goncharov's complex for weight $n$=3 and the variants of Cathelineau's complexes for weight $n$=2,3. Here we are putting all complexes together to see a bigger picture.
\end{abstract}
\section{Introduction:}
\paragraph{}
Grassmannian complex with geometric configurations is first defined by A. A. Suslin (\cite{Sus1}) and morphisms between this and Bloch-Suslin complex are introduced by A. B. Goncharov (\cite{Gonc},\cite{Gonc1}). He also introduces the morphisms of latter complexes with motivic complexes (the groups involve in this complex are classical polylogarithmic groups) for generic and degenerate configurations (see \cite{Gonc} and \cite{Gonc1}). On the other hand Cathelineau describes the infinitesimal version of motivic complexes and $F$-vector spaces which involve the groups of infinitesimal polylogarithms(see \cite{Cath1})

We suggest variants of Cathelineau's complex and describe their relation through morphisms with Grassmannian complex in configurations (see \cite{GandS}). In this article we are using Lemma \ref{v-gangl}, and we see that its exploitation will give some useful relations between the three complexes discussed above.

\section{Preliminaries and Background:} 
\paragraph{}
Throughout the note $F$ means field with characteristic zero and $F^{\bullet\bullet}=F-\{0,1\}$. In this section we will cover some primary concepts which are essential for this paper.
\subsection{Grassmannian Complex:}\label{Grass_comp}
\paragraph{}

For a set $X$, let $(\tilde{x}_1,\ldots,\tilde{x}_n)\in X^n$ for all $\tilde{x}_i\in X$, take the modulo group action of $G$ on $X^n$ then we will have $(x_1,\ldots,x_n)\in X^n$ for all $x_i\in X$ and it is called as the configuration of $n$-points in $X$. We replace $X$ by $V_n$ where $V_n$ is an $n$-dimensional vector space over the field $F$. We take the group action of $GL_n(F)$ on $V_n$ then $(x_1,\ldots,x_m)$ represents the configuration of $m$-vectors in $n$-dimensional vector space $V_n$. 

Introducing the free abelian group $C_{m,n}$ generated by the configuration of $m$-vectors $(x_1,\ldots,x_m)$ in $n$-dimensional vector space. Further define a differential map $d$ as
$$d:C_{m,n}\rightarrow C_{m-1,n}$$
by 
$$d:(x_1,\ldots,x_m)\mapsto\sum_{i=1}^m(-1)^i(x_1,\ldots,\hat{x_i},\ldots,x_m)$$
thus we define a complex $(C_{*,n},d)$ called Grassmannian complex with the degree of $d$ is -1.(see \cite{GandS},\cite{Gonc} and \cite{Gonc1})
\subsection{Polylogathmic Groups(Classical):}\label{poly_class}
\paragraph{} Define $\mathcal{B}_1(F):= F^\times$
\subsubsection{Weight 2:}
\paragraph{} First define $R_2(F)\subset \varmathbb{Z}[F]$ generated by  $$\sum_{i=0}^4(-1)^i\left[r(x_0,\ldots\hat{x}_i\ldots,x_4)\right]$$
where $r(x_0,\ldots,x_3)=\frac{\Delta(x_0,x_3)
\Delta(x_1,x_2)}{\Delta(x_0,x_2)\Delta(x_1,x_3)}$ is the cross-ratio of 4-points and $x_0,\ldots,x_3\in F$ and $\Delta(x_i,x_j)$ is the $ 2\times 2$ determinant.

Now introducing $\mathcal{B}_2(F)$, which is defined as the quotient of $\varmathbb{Z}[F]$ by $R_2(F)$ that is
$$\mathcal{B}_2(F):=\frac{\varmathbb{Z}[F]}{R_2(F)}$$
This group can be put in the famous Bloch-Suslin complex $B_F(2):\mathcal{B}_2(F)\xrightarrow{\delta} \wedge^2 F^\times$
where $\wedge^2F^\times=F^\times\otimes F^\times/\langle x\otimes x|x\in F^\times\rangle$ and $\delta([a]_2)=(1-a)\wedge a$
\subsubsection{Weight 3:}
\paragraph{}
First define the map $r_3:C_{6,3}\rightarrow \varmathbb{Z}[F]$ by 
$$r_3(x_0,\ldots,x_5)=\text{Alt}_6\left[
\frac{\Delta(x_0,x_1,x_3)\Delta(x_1,x_2,x_4)
\Delta(x_2,x_0,x_5)}{\Delta(x_0,x_1,x_4)
\Delta(x_1,x_2,x_5)\Delta(x_2,x_0,x_3)}
\right]$$
$R_3(F)\subset \varmathbb{Z}[F]$ is generated by seven-term relation $\sum_{i=0}^6(-1)^ir_3(x_0,\ldots,\hat{x}_i,\ldots,x_6)$ which is a big result and was very helpful to prove Zagier's conjecture for weight 3 (see \cite{Gonc1}).
Thus Goncharov define the weight 3 motivic group $\mathcal{B}_3(F)$ is the quotient group of $\varmathbb{Z}[F]$ by $R_3(F)$, that is
$$\mathcal{B}_3(F)=\frac{\varmathbb{Z}[F]}{R_3(F)}$$
These groups can be placed in the sequence to form the Goncharov's complex of weight 3.
$$B_F(3):\mathcal{B}_3(F)\xrightarrow{\delta}
\mathcal{B}_2(F)\otimes F^\times\xrightarrow{\delta}\wedge^3F^
\times$$
where $\delta([a]_3)=[a]_2\otimes a$ and $\delta([a]_2\otimes b)=(1-a)\wedge a\wedge b$
\subsection{Polylogarithmic Groups(Infinitesimal):}\label{poly_inf}
\paragraph{}
In fact here we will discuss $F$-vector spaces, but before that it is necessary to define general derivation in $F$-vector spaces. Let $D$ be the derivation map $D\in Der_{\varmathbb{Z}}(F,F)$ which means if $a\in F$ then $D(a)\in F$ and satisfies the following conditions.
\begin{enumerate}
\item $D(a+b)=D(a)+D(b)$ 
\item $D(ab)=aD(b)+bD(a)$ 
\end{enumerate}
for all $a, b, D(a), D(b)\in F$

Cathelineau (\cite{Cath1},\cite{Cath2}) has defined the infinitesimal versions of Goncharov's groups $\mathcal{B}_n(F)$

\subsubsection{Weight 2:}
\paragraph{}
Let $\langle a\rangle$ generates the vector space $F[F]$ over the field $F$ and we set $\llbracket a \rrbracket^D :=\frac{D(a)}{a(1-a)}\langle a\rangle$, we define another quotient space $\beta^D_2(F)$ of $F[F]$ by $\rho_2^D(F)$, where $\rho_2^D(F)\subset F[F]$ generated by the following five term relation.
$$\llbracket a \rrbracket^D-\llbracket b \rrbracket^D+\left\llbracket \frac{b}{a} \right\rrbracket^D-\left\llbracket \frac{1-b}{1-a} \right\rrbracket^D+\left\llbracket \frac{1-b^{-1}}{1-a^{-1}} \right\rrbracket^D$$
This vector space $\beta^D_2(F)$ is the variant of Cathelineau's vector $\beta_2(F)$ (see \cite{GandS}, \cite{Cath1}, \cite{Cath2}). We construct a complex 
$$\beta^D_2(F)\xrightarrow{\partial^D} F\otimes F^\times$$
where $\partial^D(\llbracket a \rrbracket^D_2)=-D\log (1-a)\otimes a+D\log a\otimes (1-a)$ and $D\log (a)=\frac{D(a)}{a}$
\subsubsection{Weight 3:}
 $F$-vector space is defined as the quotient space of $F[F]$ by $\rho^D_3(F)$, that is 
$$\beta^D_3(F)=\frac{F[F]}{\rho^D_3(F)}$$
where $\rho^D_3(F)$ is the kernel of the following map
$$\partial^D:F[F]\rightarrow \beta^D_2(F)\otimes F^\times\oplus F\otimes\mathcal{B}_2(F)$$
$$\partial^D(\llbracket a \rrbracket^D_3)=\llbracket a \rrbracket^D_2\otimes a+(1-a)\otimes [a]_2$$
we can make a complex by using the above vector spaces.
$$\beta^D_3(F)\xrightarrow{\partial^D}
\beta^D_2(F)\otimes F^\times\oplus F\otimes\mathcal{B}_2(F)\xrightarrow{
\partial^D} F\otimes \wedge^2F^\times$$
\begin{lem}\label{v-gangl}(Elbaz-Vincent--Gangl)(see Lemma 6.1 and Proposition 6.2 of \cite{PandG})
Let $D\in Der_{\varmathbb{Z}}(F)$ be an absolute derivation for the field $F$. Then the following diagram is commutative.
\begin{displaymath}
\xymatrix{
\varmathbb{Z}[F^{\bullet\bullet}]\ar[rr]^{f_D}\ar[d]^
{\delta_n}&&F[F^{\bullet\bullet}]\ar[d]^
{\partial_n}\\
\mathcal{B}_{n-1}(F)\otimes F^\times\ar[rr]^{g^n_D\qquad\qquad}&&\left
(\beta_{n-1}^D(F)\otimes F^\times\right) \oplus \left(F\otimes\mathcal{B}_{n-1}(F)\right)}\\
\end{displaymath}
Where $f_D:[a]\mapsto\frac{D(a)}{a(1-a)}[a]$ 
\begin{equation*}
\delta_n:[a]\mapsto
 \begin{cases}
  [a]_{n-1}\otimes a \quad &\text{ for } n>2\\
  (1-a)\wedge a \quad &\text{ for } n=2\\
 \end{cases}
\end{equation*}
\begin{equation*}
\partial_n:[a]\mapsto
 \begin{cases}
  \llbracket a \rrbracket^D_{n-1}\otimes a + (1-a)\otimes [a]_{n-1}\quad &\text{ for } n>2\\
  -\frac{D(a)}{1-a}\otimes a+\frac{D(a)}{a}\otimes (1-a)\quad &\text{ for } n=2\\
 \end{cases}
\end{equation*}
\[g^n_D:[a]_{n-1}\otimes b\mapsto\llbracket a\rrbracket^D_{n-1}\otimes b+\frac{D(b)}{b}\otimes [a]_{n-1}\]
\end{lem}
\begin{proof}
Proof requires direct computation(see the proof of Lemma 6.1 and Proposition 6.2 in \cite{PandG}).  
\end{proof}

\section{Morphisms in weight 2 complexes:}

Recall the diagram (2.9) of $\S$2.3 in \cite{Gonc}. Goncharov proved that the following diagram is commutative.
\begin{displaymath}\label{sq_1}
\xymatrix{
C_{4,2}\ar[r]^d\ar[d]^{f^2_1}&C_{3,2}\ar[d]
^{f^2_0}\\
\mathcal{B}_2(F)\ar[r]^{\delta}&\bigwedge^
2F^\times}\tag{3.1a }\\
\end{displaymath}
for the following maps 
\[f^2_0(x_0,x_1,x_2)=\Delta(x_0,x_1)\wedge
\Delta(x_0,x_2)-\Delta(x_0,x_1)\wedge
\Delta(x_1,x_2)+\Delta(x_0,x_2)\wedge
\Delta(x_1,x_2)\]
and
\[f^2_1(x_0\ldots,x_3)=\left[\frac{\Delta
(x_0,x_3)\Delta(x_1,x_2)}{\Delta(x_0,x_2)\Delta(x_1,x_3)}\right]_2\]
where $d$ is defined in $\S$\ref{Grass_comp} and $\delta$ is defined in $\S$\ref{poly_class} above.
\begin{lem}
The diagram (\ref{sq_1}) is commutative, i.e. $f^2_0\circ d=\delta\circ f^2_1$
\end{lem}
\begin{proof}
Proof requires direct computations (see \cite{Gonc},\cite{Gonc1} for the proof)
\end{proof}
In \cite{GandS}, we introduced similar diagram which provide the morphisms between Grassmannian and the variant of Cathelineau's complex for weight 2. 
\begin{displaymath}\label{sq_2}
\xymatrix{
C_{4,2}\ar[r]^d\ar[d]^{\tau_1^{2}}& C_{3,2}\ar[d]^{\tau_0^{2}}\\ \beta^D_2(F)\ar[r]^{\partial^D}              & F\otimes F^\times }\tag{3.1b}\
\end{displaymath}
where $\beta^D_2(F)$ and $\partial^D$ are defined in $\S$\ref{poly_inf}, we define 
\begin{align}\label{todef}
\tau_0^{2}\colon(x_0,x_1,x_2)\mapsto &\sum_{i=0}^2 \Bigg\{\frac{D\{\Delta(x_i,x_{i+2})\}}{\Delta(x_i,x_{i+2})}\otimes \Delta(x_i,x_{i+1})\notag\\
&\quad-\frac{D\{\Delta(x_{i+1},x_i)\}}{\Delta(x_{i+1},x_i)}\otimes \Delta(x_i,x_{i+2})\Bigg\} \quad\text{$i$ mod 3}\notag
\end{align}
\begin{equation}\label{t1def}
\tau_1^{2}\colon (x_0,\ldots,x_3)\mapsto \llbracket r(x_0,\ldots,x_3)\rrbracket^D_2\notag
\end{equation}
\begin{lem}
$\tau_0^2\circ d=\partial^D\circ \tau^2_1$
\end{lem}
\begin{proof}
See \cite{GandS} for the proof.
\end{proof}
Lemma \ref{v-gangl} and diagrams (\ref{sq_1}) and (\ref{sq_2}) will give us the following prism like diagram which has six faces and the next result in this section will show that all faces of the following diagram are commutative.

\begin{displaymath}\label{tri2D}
 \xymatrix{
& C_{4,2}\ar[rr]^d\ar[dd]\ar[ld]_{f_1^2}  &&C_{3,2}\ar[dd]^{\tau_{0}^{2}}\ar[dl]_
{f_0^2}\\
\mathcal{B}_2(F)\ar[rr]\ar[dr]^{\tau^2_D} &\ar[d]^{\tau_{1}^{2}} &\bigwedge^2F^\times\ar[dr]^{g^1_{2,D}}&\\
& \beta^D_2(F)\ar[rr]^{\partial} && F\otimes F^\times }\tag{3.2b} \\
\end{displaymath}
where (see $\S$6.1 in \cite{PandG})
\[\tau^2_D([x]_2)=\llbracket x\rrbracket_2\quad\text{and}\quad g^1_{2,D}(x\wedge y)=\frac{D(x)}{x}\otimes y-\frac{D(y)}{y}\otimes x\]
\begin{cor}
The diagram \eqref{tri2D} above is commutative, i.e. there is a morphism of complexes between all three complexes used in diagram \eqref{tri2D}.
\end{cor}
\begin{proof}
We only need to show that $g^1_{2,D}\circ f^2_0(x_0,x_1,x_3)=\tau^2_0(x_0,x_1,x_3)$ and $\tau^2_D\circ f^2_1(x_0,\ldots,x_3)=\tau^2_1(x_0,\ldots,x_3)$.
\begin{align*}
g^1_{2,D}\circ f^2_0(x_0,x_1,x_3)=&g^1_{2,D}
\left(\Delta(x_0,x_1)\wedge\Delta(x_0,x_2)-
\Delta(x_0,x_1)\wedge\Delta(x_1,x_2)+
\Delta(x_0,x_2)\wedge\Delta(x_1,x_2)\right)
\\
=&\frac{D\left(\Delta(x_0,x_1)\right)}{\Delta(x_0,x_1)}\otimes \Delta(x_0,x_2)-\frac{D
\left(\Delta(x_0,x_2)\right)}{\Delta(x_0,x_2)}\otimes \Delta(x_0,x_1)\\
-&\frac{D\left(\Delta(x_0,x_1)\right)}{\Delta(x_0,x_1)}\otimes \Delta(x_1,x_2)+\frac{D
\left(\Delta(x_1,x_2)\right)}{\Delta(x_1,x_2)}\otimes \Delta(x_0,_1)\\
+&\frac{D\left(\Delta(x_0,x_2)\right)}{\Delta(x_0,x_2)}\otimes \Delta(x_1,x_2)-\frac{D
\left(\Delta(x_1,x_2)\right)}{\Delta(x_1,x_2)}\otimes \Delta(x_0,x_2)\\
=&\tau^2_0(x_0,x_1,x_3)
\end{align*}
and
\begin{align*}
\tau^2_D\circ f^2_1(x_0,\ldots,x_3)=&\tau^2_D\left(\left[\frac{
\Delta(x_0,x_3)
\Delta(x_1,x_2)}{\Delta(x_0,x_2)\Delta(x_1,x_3)}\right]_2
\right)=\left\llbracket\frac{\Delta(x_0,x_3)
\Delta(x_1,x_2)}{\Delta(x_0,x_2)\Delta(x_1,x_3)}
\right\rrbracket_2^D
\end{align*}
\end{proof}

\section{Morphisms in Weight 3 complexes:}
\paragraph{}
Similar to weight $n$=2, Goncharov constructed morphisms between Grassmannian and weight $n$=3 motivic complexes. Diagram (3.2) in $\S$3 of \cite{Gonc1} describes about the morphisms.
\begin{displaymath}\label{Gonc2d}
\xymatrix{
C_{6,3}\ar[r]^{d}\ar[d]^{f_2^3}     & C_{5,3}\ar[r]^{d}\ar[d]^{f_1^{3}}        &C_{4,3}\ar[d]^{f_0^{3}}\\
\mathcal{B}_3(F)\ar[r]^{\delta\quad}     & \mathcal{B}_2(F)\otimes F^\times \ar[r]^{\quad\delta}          & \bigwedge^3 F^\times}\tag{4.1a}\ 
\end{displaymath}
is commutative for the following maps
\[f^3_0(x_0,\ldots,x_3)=\sum_{i=0}^3(-1)^i
\bigwedge_{\substack{j=0\\j\neq i}}^3\Delta(x_0,\ldots,\hat{x}_j,\ldots,x_3),\]
\[f^3_1(x_0,\ldots,x_4)=-\frac{1}{3}\sum_{i=0}^4(-1)^i\left[r(x_i|x_0,\ldots,\hat{x}_i,\ldots,x_4)\right]_2\otimes \prod_{j\neq i}\Delta(\hat{x}_i,\hat{x}_j)\]
and $f^3_2$ is defined via alternation sum for generic points.
\[f^3_2(x_0,\ldots,x_5)=\frac{2}{45}\text{Alt}_6\left[\frac{\Delta(x_0,x_1,
x_3)\Delta(x_1,x_2,x_4)\Delta(x_2,x_0,x_5)}{\Delta(x_0,x_1,x_4)\Delta(x_1,x_2,x_5)
\Delta(x_2,x_0,x_3)}\right]_3 \]
where $\delta\left([x]_3\right)=[x]_2\otimes x$ for all $x\neq0,1\in F^\times$ and $\delta([x]_2)=(1-x)\wedge x$.
\begin{thm}
Diagram (\ref{Gonc2d}) is commutative, i.e.
\begin{enumerate}
\item $f^3_0\circ d=\delta\circ f^3_1$  \item $f^3_1\circ d=\delta\circ f^3_2$
\end{enumerate}
\end{thm}
\begin{proof}
See \cite{Gonc1} for (1) and appendix of \cite{Gonc4} for the proof of (2). 
\end{proof}

In \cite{GandS}, we introduced another diagram which gives the morphisms between Grassmannian and the variant of cathelineau's complex for weight $n$=3.
\begin{displaymath}\label{bicomp2}
\xymatrix{
C_{6,3}\ar[r]^{d}\ar[d]^{\tau_2^3}     & C_{5,3}\ar[r]^{d}\ar[d]^{\tau_1^{3}}        &C_{4,3}\ar[d]^{\tau_0^{3}}\\
\beta^D_3(F)\ar[r]^{\partial\qquad\qquad}     & (\beta^D_2(F)\otimes F^\times) \oplus  (F\otimes \mathcal{B}_2(F))\ar[r]^{\qquad\qquad\partial}          & F\otimes \bigwedge^2 F^\times}\ \tag{4.2a}
\end{displaymath}
where 
\begin{align} \label{t03def}
\tau_0^{3}:(x_0,\ldots,x_3)\mapsto\sum_{i=0}^3(-1)^{i}&
\frac{D\Delta(x_0,\ldots,\hat{x_i},\ldots,x_3)}{\Delta(x_0,\ldots,\hat{x}_i,\ldots,x_3)}\otimes \frac{\Delta(x_0,\ldots,\hat{x}_{i+1},\ldots,x_3)}{\Delta(x_0,\ldots,\hat{x}_{i+2},\ldots,x_3)}\notag\\ &\wedge\frac{\Delta(x_0,\ldots,\hat{x}_{i+3},\ldots,x_3)}{\Delta(x_0,\ldots,\hat{x}_{i+2},\ldots,x_3)} 
\end{align}
\begin{align}
\tau_1^{3}:(x_0,\ldots,x_4)\mapsto -\frac{1}{3}\sum_{i=0}^4 &(-1)^i\Big\{
\llbracket r(x_i|l_0,\ldots,\hat{x}_i,\ldots,x_4) \rrbracket^D_2 \otimes \prod_{j\neq i}\Delta(\hat{x}_i,\hat{x}_j) \notag \\
&+\frac{D\left(\prod_{j\neq i}\Delta(\hat{x}_i,\hat{x}_j)\right)}{\prod_{j\neq i}\Delta(\hat{x}_i,\hat{x}_j)}\otimes [r(x_i|x_0,\ldots,\hat{x}_i,\ldots,x_4)]_2\Big\} \notag
\end{align}
\begin{align}
&\tau_2^3:(x_0,\ldots,x_5)\mapsto\frac{2}{45}\text{Alt}_6\left\llbracket \frac{\Delta(x_0,x_1,x_3)\Delta(x_1,x_2,x_4)\Delta(x_2,x_0,x_5)}{\Delta(x_0,x_1,x_4)\Delta(x_1,x_2,x_5)\Delta(x_2,x_0,x_3)}\right\rrbracket^D_3\notag
\end{align}
where \[\llbracket a\rrbracket^D_3 = \frac{D(a)}{a(1-a)}\langle a\rangle_3\text{ and } \Delta(\hat{x}_i,\hat{x}_j)=\Delta(x_0,\ldots,\hat{x}_i,\ldots,\hat{x}_j,\ldots,x_4)\]
\[\partial^D\left(\left\llbracket a\right\rrbracket^D_3\right) =\left\llbracket a \right\rrbracket^D_2\otimes a +\frac{D(a)}{a}\otimes [a]_2\]
\[\partial^D\left(\llbracket a\rrbracket^D_2\otimes b + x\otimes [y]_2\right)=\frac{D(1-a)}{1-a}\otimes a\wedge b-\frac{D(a)}{a}\otimes (1-a)\wedge b+x\otimes (1-y)\wedge y\]
\begin{thm}
Diagram (\ref{bicomp2}) is commutative, i.e.,
\begin{enumerate}
\item $\tau^3_0\circ d=\partial^D\circ\tau^3_1$
\item $\tau^3_1\circ d=\partial^D\circ\tau^3_2$
\end{enumerate}
\end{thm}
\begin{proof}

\end{proof}
So, we can combine diagrams (\ref{Gonc2d}) and (\ref{bicomp2}), again exploit Lemma \ref{v-gangl} to get the following diagram. We will show that all the faces of the following diagram is commutative. \begin{displaymath}\label{tri3D}
 \xymatrix{
&C_{6,3}\ar[r]^d\ar[dd]\ar[ld]_{f_{2}^3}&
C_{5,3}\ar[r]^d\ar[dd]\ar[ld]_{f_{1}^3}  &C_{4,3}\ar[dd]^{\tau_{0}^{3}}
\ar[dl]_{f_{0}^3}\\
\mathcal{B}_3(F)\ar[r]^{\delta}
\ar[dr]^{\tau^3_{D}} &\ar[d]^{\tau^3_{2}}\mathcal{B}_2(F)\otimes F^\times\ar[r]^{\delta}\ar[dr]^{g^2_{3,D}} &\ar[d]^{\tau_{1}^{3}} \bigwedge^3F^\times\ar[dr]^{g^1_{3,D}}\\
&\beta^D_3(F)\ar[r]^{\partial^D\qquad\qquad} & \left(\beta^D_2(F)\otimes F^\times\right)\oplus \left(F\otimes \mathcal{B}_2(F)\right)\ar[r]^{\qquad\qquad\partial^D}& F\otimes \bigwedge^2F^\times }\tag{4.2b} \\
\end{displaymath}
where (see $\S$6.1 in \cite{PandG})
\[\tau^3_D([x]_3)=\llbracket x \rrbracket_3,\quad g^2_{3,D}([ x]_2\otimes y)=\llbracket x\rrbracket_2\otimes y+ \frac{D(y)}{y}\otimes [x]_2\]
and
\[g^1_{3,D}(x\wedge y\wedge z)=\frac{D(x)}{x}\otimes y\wedge z-\frac{D(y)}{y}\otimes x\wedge z+\frac{D(z)}{z}\otimes x\wedge y\]
Note that all the morphisms here are well-defined and the well-definedness of $f_0^2,f_1^2,f_0^3,f_1^3,f_2^3$ were shown in \cite{Gonc1} while the well-definedness of $\tau_0^2,\tau_1^2,\tau_0^3,\tau_1^3,\tau_2^3$ were shown in \cite{GandS}.

\begin{cor}
The diagram \eqref{tri3D} above is commutative, i.e. there is a morphism of complexes between all three complexes used in diagram \eqref{tri3D}.
\end{cor}
\begin{proof}
We only need to show that $g^1_{3,D}\circ f^3_0(x_0,\ldots,x_3)=\tau^3_0(x_0,\ldots,x_3)$, $g^2_{3,D}\circ f^3_1(x_0,\ldots,x_4)=\tau^3_1(x_0,\ldots,x_4)$ and $\tau^3_D\circ f^3_2(x_0,\ldots,x_5)=\tau^3_2(x_0,\ldots,x_5)$
\begin{align*}
g^1_{3,D}&\circ f^3_0(x_0,\ldots,x_3)\\
=&g^1_{3,D}\left(\sum_{i=0}^3(-1)^i
\bigwedge_{\substack{j=0\\j\neq i}}^3\Delta(x_0,\ldots,\hat{x}_j,\ldots,x_3)\right)\\
=&\sum^3_{i=0}(-1)^i\left(\frac{D\left(\Delta(x_0,\ldots,\hat{x}_i,\ldots,x_4)\right)}{\Delta(x_0,\ldots,\hat{x}_i,\ldots,x_4)}\otimes\frac{\Delta(x_0,\ldots,\hat{x}_{i+1},\ldots,x_4)}{\Delta(x_0,\ldots,\hat{x}_{i+2},\ldots,x_4)}\wedge\frac{\Delta(x_0,\ldots,\hat{x}_{i+2},\ldots,x_4)}{\Delta(x_0,\ldots,\hat{x}_{i+3},\ldots,x_4)}\right),\\
&\qquad\qquad\qquad\qquad\qquad\qquad\qquad
\qquad\qquad\qquad\qquad\qquad\qquad\qquad
\quad i \mod 4
\end{align*}
\begin{align*}
g^2_{3,D}&\circ f^3_1(x_0,\ldots,x_4)\\
=&g^2_{3,D}\left(\sum^4_{i=0}(-1)^i\left[r\left(x_i|x_0,\ldots,\hat{x}_i,\ldots,x_4\right)\right]_2\otimes \prod_{i\neq j}\Delta(\hat{x}_i,\hat{x}_j)\right)\\
=&\sum^4_{i=0}(-1)^i\left\llbracket r\left(x_i|x_0,\ldots,\hat{x}_i,\ldots,x_4\right)\right\rrbracket_2\otimes \prod_{i\neq j}\Delta(\hat{x}_i,\hat{x}_j)\\
+&\sum^4_{\substack{j=0\\j\neq i}}\frac{D(\hat{x}_i,\hat{x}_j)}{(\hat{x}_i,\hat{x}_j)}\otimes\left[r\left(x_i|x_0,\ldots,\hat{x}_i,\ldots,x_4\right)\right]_2
\end{align*}
\begin{align*}
\tau^3_D\circ f^3_2(x_0,\ldots,x_5)=&\tau^3_D\left(\frac{2}{45}\text{Alt}_6\left[\frac{
\Delta(x_0,x_1,x_3)\Delta(x_1,x_2,x_4)
\Delta(x_2,x_0,x_5)}{\Delta(x_0,x_1,x_4)
\Delta(x_1,x_2,x_5)\Delta(x_2,x_0,x_3)}
\right]_3\right)\\
=&\frac{2}{45}\text{Alt}_6\left\llbracket\frac{
\Delta(x_0,x_1,x_3)
\Delta(x_1,x_2,x_4)\Delta(x_2,x_0,x_5)}{\Delta(x_0,x_1,x_4)\Delta(x_1,x_2,x_5)
\Delta(x_2,x_0,x_3)}\right\rrbracket_3
\end{align*}
\end{proof}

\section*{Acknowledgement}
The author would like to thanks Herbert Gangl, Victor Abrashkan and Tony Scholl. This paper consists on one chapter of the author's doctoral thesis at University of Durham.

\end{document}